\documentclass{sig-alternate}

\usepackage[%
  breaklinks=true,%
  colorlinks=true,%
  linkcolor=blue,anchorcolor=blue,%
  citecolor=blue,filecolor=blue,%
  menucolor=blue,%
  urlcolor=blue]{hyperref}
\usepackage{listings}
\usepackage{amssymb} 
\usepackage{amsfonts} 
\usepackage[linesnumbered,ruled]{algorithm2e}

\usepackage{doi}

\usepackage{pgfplots,tikz}
\usetikzlibrary{plotmarks}
\pgfplotsset{compat=newest}
\pgfplotsset{plot coordinates/math parser=false}
\newlength\figureheight
\newlength\figurewidth

\newtheorem{theorem}{Theorem}

\usepackage{subfig}
\usepackage{scalefnt}

\title{Nonsymmetric multigrid preconditioning for conjugate gradient methods}

\numberofauthors{3}
\author{
\alignauthor
Henricus Bouwmeester\\
\affaddr{University of Colorado Denver}\\
\affaddr{P.O. Box 173364}\\
\affaddr{Campus Box 170}\\ 
\affaddr{Denver, CO 80217-3364}\\
\email{\small Henricus.Bouwmeester@ucdenver.edu}
\alignauthor
Andrew Dougherty\\
\affaddr{University of Colorado Denver}\\
\affaddr{P.O. Box 173364}\\
\affaddr{Campus Box 170}\\ 
\affaddr{Denver, CO 80217-3364}\\
\email{\small Andrew.Dougherty@ucdenver.edu}
\alignauthor
Andrew V. Knyazev\\
\affaddr{Mitsubishi Electric Research Laboratories}\\ 
\affaddr{201 Broadway}\\
\affaddr{Cambridge, MA 02145}\\
\email{\small Andrew.Knyazev@merl.com}
}

\begin{document}
\maketitle

\begin{abstract}
We numerically analyze the possibility of turning off post-smoothing
(relaxation) in geometric multigrid when used as a preconditioner in conjugate
gradient linear and eigenvalue solvers for the 3D Laplacian.  The geometric
Semicoarsening Multigrid (SMG) method is provided by the \emph{hypre} parallel
software package.   We solve linear systems using two variants (standard and
flexible) of the preconditioned conjugate gradient (PCG) and preconditioned
steepest descent (PSD) methods. The eigenvalue problems are solved using the
locally optimal block preconditioned conjugate gradient (LOBPCG) method
available in \emph{hypre} through BLOPEX software.  We observe that turning off
the post-smoothing in SMG dramatically slows down the standard PCG-SMG. For
flexible PCG  and LOBPCG, our numerical results show that post-smoothing can be
avoided, resulting in overall acceleration, due to the high costs of smoothing
and relatively insignificant decrease in convergence speed.    We numerically
demonstrate for linear systems that PSD-SMG and flexible PCG-SMG converge
similarly if SMG post-smoothing is off.  We experimentally show that the effect
of acceleration is independent of memory interconnection.  A theoretical
justification is provided. 
\end{abstract}


\section{Introduction}
\label{sec:intro} 
Smoothing (relaxation) and coarse-grid correction are the two cornerstones of
multigrid technique. In algebraic multigrid, where only the system matrix is
(possibly implicitly) available, smoothing is more fundamental since it is often
used to construct the coarse grid problem. In geometric multigrid, the coarse
grid is generated by taking into account the geometry of the fine grid, in
addition to the chosen smoothing procedure. If full multigrid is used as a
stand-alone solver, proper smoothing is absolutely necessary for convergence.
If multigrid is used as a preconditioner in an iterative method, one is tempted
to check what happens if smoothing is turned partially off.
\let\thefootnote\relax\footnote{A prepreint is available at
http://arxiv.org/abs/1212.6680}

For symmetric positive definite (SPD) linear systems, the preconditioner is
typically required to be also a fixed linear SPD operator, to preserve the
symmetry of the preconditioned system.  In the multigrid context, the
preconditioner symmetry is achieved by using balanced pre- and post-smoothing,
and by properly choosing the restriction and prolongation pair.  In order to get
a fixed linear preconditioner, one  avoids using nonlinear smoothing,
restriction, prolongation, or coarse solves.  The positive definiteness is
obtained by performing enough (in practice, even one may be enough), and an
equal number of, pre- and post-smoothing steps; see,
e.g.,~\cite{Bramble2000173}.
 
If smoothing is unbalanced, e.g.,\ there is one step of pre-smoothing, but no
post-smoothing,  the multigrid preconditioner becomes nonsymmetric.  Traditional
assumptions of the standard convergence theory of iterative solvers are no
longer valid, and convergence behavior may be unpredictable. The main goal of
this paper is to describe our numerical experience experimenting with the
influence of unbalanced smoothing in  practical geometric multigrid
preconditioning, specifically, the  Semicoarsening Multigrid (SMG) method,
see~\cite{schaffer305237}, provided by the parallel software package
\emph{hypre}~\cite{hypreUserManual}.  

We numerically analyze the possibility of turning off post-smoothing in
geometric multigrid when used as a preconditioner in iterative linear and
eigenvalue solvers for the 3D Laplacian in \emph{hypre}.   We solve linear
systems using two variants (standard and flexible, e.g.,\ \cite{golub1305}) of
the preconditioned conjugate gradient (PCG) and preconditioned steepest descent
(PSD) methods.  The standard PCG is already coded in \emph{hypre}. We have
written the codes of flexible PCG and PSD by modifying the \emph{hypre} standard
PCG function.  The eigenvalue problems are solved using the locally optimal
block preconditioned conjugate gradient (LOBPCG) method, readily available in
\emph{hypre} through BLOPEX~\cite{knyazev2224}.

We observe that turning off the post-smoothing in SMG dramatically slows down
the standard PCG-SMG.  However, for the flexible PCG  and LOBPCG, our numerical
tests show that post-smoothing can be avoided.  In the latter case, turning off
the post-smoothing in SMG results in overall acceleration, due to the high costs
of smoothing and relatively insignificant decrease in convergence speed.  Our
observations are also generally applicable for algebraic multigrid
preconditioning for graph Laplacians, appearing, e.g., in computational
photography problems, as well as in 3D mesh processing
tasks~\cite{krishnan:191409}.

A different case of non-standard preconditioning, specifically, \emph{variable
preconditioning}, in PCG is considered in our earlier work~\cite{knyazev1267}.
There, we also find a dramatic difference in convergence speed between the
standard and flexible version of PCG.  The better convergence behavior of the
flexible PCG is explained in~\cite{knyazev1267} by its \emph{local optimality},
which guarantees its convergence with at least the speed of PSD.  Our numerical
tests there show that, in fact, the convergence of PSD and the flexible PCG is
practically very close.  We perform the same comparison here, and obtain a
similar result.    We demonstrate for linear systems that PSD-SMG converges
almost as fast as the  flexible PCG-SMG if SMG post-smoothing is off in both
methods.  

Our numerical experiments are executed in both a strictly shared memory
environment and in a distributed memory environment.  Different size problems
and with different shapes of the bricks are solved in our shared versus
distributed tests, so the results are not directly comparable with each other.
Our motivation to test both shared and distributed cases is to investigate how
the effect of acceleration depends on the memory interconnection speed.

The rest of the paper is organized as follows.  We formally describe the PSD and
PCG methods used here for testing, and explain their differences.  We briefly
discuss the SMG preconditioning in \emph{hypre} and present our numerical
results for linear systems. One section is dedicated to eigenvalue problems. Our
last section contains the relevant theory. 

\section{PSD and PCG methods}
\label{sec:PCG}   
For a general exposition of PSD and PCG, let  
SPD matrices $A$ and $T$, and vectors $b$ and $x_0$ be given, and denote $r_k = b - Ax_k$.  
Algorithm~\ref{alg:CG} is described 
in~\cite{knyazev1267}  

\begin{algorithm}
    \For{$k=0,1,\ldots$}{
    $s_k = T r_k$\\
    \eIf{$k=0$}{
      $p_0 = s_0$\\
    }
    {
    $p_k = s_k + \beta_k p_{k-1}$ (where $\beta_k$ is 
    either~(\ref{eqn:beta_orig}) or~(\ref{eqn:beta_alt}) for all iterations)\\
    }
    $\displaystyle \alpha_k = \frac{(s_k,r_k)}{(p_k,Ap_k)}$\\
    $x_{k+1} = x_k + \alpha_k p_k$\\
    $r_{k+1} = r_k - \alpha_k Ap_k$\\
    }
    \caption{PSD and PCG methods}
    \label{alg:CG}
\end{algorithm}

Various methods are obtained by using different formulas for the scalar $\beta_k$. 
We set $\beta_k=0$ for PSD, 
\begin{equation}
    \beta_k = \frac{(s_k,r_k)}{(s_{k-1},r_{k-1})} 
\label{eqn:beta_orig}
\end{equation}
for the standard PCG, or 
\begin{equation}
    \beta_k = \frac{(s_k,r_k-r_{k-1})}{(s_{k-1},r_{k-1})}
\label{eqn:beta_alt}
\end{equation}
for the flexible PCG.

We note that in using~(\ref{eqn:beta_alt}), we are merely subtracting one term, 
$(s_k,r_{k-1})$, 
in the numerator of~(\ref{eqn:beta_orig}), which appears in the standard CG algorithm.
If $T$ is a fixed SPD matrix, this term actually vanishes; see, e.g.,\ \cite{knyazev1267}. 
By using~(\ref{eqn:beta_alt}) in a computer code, 
it is required that an extra vector be allocated to either calculate $r_k - r_{k-1}$ 
or store $-\alpha_k Ap_k$, compared to~(\ref{eqn:beta_orig}). 
The associated costs may be noticeable for large problems solved on 
parallel computers. 
Next, we numerically evaluate the extra costs by comparing the 
standard and flexible PCG 
with no preconditioning 
for a variety of problem sizes. 

Our model problem used for all calculations in the present paper is for the
three-dimensional negative Laplacian in a brick with homogeneous Dirichlet
boundary conditions approximated by the standard finite difference scheme using
the 7-point stencil with the grid size one in all three directions. The initial
approximation is (pseudo)random.  We simply call a code, called \emph{struct},
which is provided in  \emph{hypre} to test SMG, with different command-line
options. 

To generate the data to compare iterative methods, we execute the following
command,
\begin{lstlisting}
mpiexec -np 16 ./struct -n $n $n $n -solver 19
\end{lstlisting}
for the shared memory experiments and 
\begin{lstlisting}
mpiexec -np 384 ./struct -n $n $n $n -solver 19
\end{lstlisting}
for the distributed memory experiments, where \$n runs from 10 to 180, and
determines the number of grid points in each of the three directions \emph{per
processor}.  The size of the brick here is \$np-times-\$n-by-\$n-by-\$n, i.e., the
brick gets longer in the first direction with the increase in the number of
cores.  For example, using the largest value \$n=180, the maximum problem size
we solve is 16x180-by-180-by-180=93,312,000 unknowns for \$np=16 and
384x180-by-180-by-180=2,239,488,000 unknowns for \$np=384.  The option \emph{-solver
19} tells the driver \emph{struct} to use no preconditioning.  The MPI option
\emph{-np~16} means that we run on 16 cores and we restrict to using only one
node for the shared memory whereas for distributed memory we use 16 cores on 24
nodes with the MPI option \emph{-np~384}. In fact, all our tests in this paper
are performed on either 16 cores on one node or 16 cores on each of the 24
nodes, so in the rest of the paper we always omit the ``mpiexec -np 16(384)''
part of the execution command for brevity.  

Our comprehensive studies are performed on a 24 node cluster where each node is
a dual-socket octo-core machine based on an Intel Xeon E5-2670 Sandy Bridge
processor running at 2.60 GHz with 64 GB RAM of shared memory per node.
Switched communications fabric between the nodes is provided by an InfiniBand 4X
QDR Interconnect.  Each experiment is performed using 16 cores on one node using
shared memory and also using 16 cores on all 24 nodes via distributed memory.

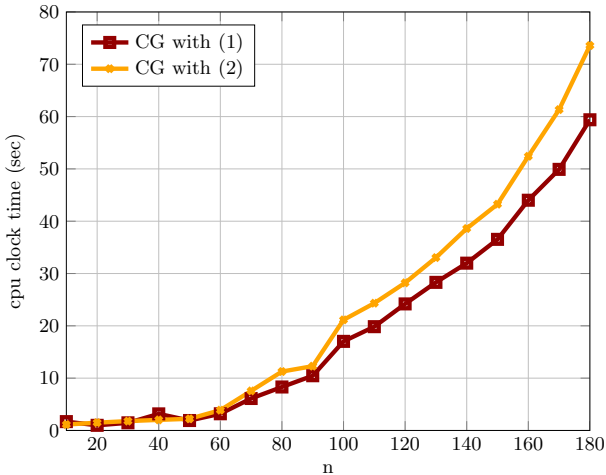
\begin{figure}[ht]
\centering
\subfloat[Shared memory]{
    \label{fig:extra_storage_costs_smp}
    \setlength\figurewidth{\linewidth} 
    \setlength\figureheight{0.80\figurewidth} 
    \hspace{-15pt}
    \resizebox{.475\textwidth}{!}{{\scalefont{1}
    
\definecolor{mycolor1}{rgb}{0.576470588235294,0,0}
\definecolor{mycolor2}{rgb}{1,0.647058823529412,0}

\begin{tikzpicture}

\begin{axis}[%
width=\figurewidth,
height=\figureheight,
scale only axis,
xmin=10,
xmax=180,
xtick={20,40,60,80,100,120,140,160,180},
xlabel={n},
xmajorgrids,
ymin=0,
ymax=80,
ytick={0,10,20,30,40,50,60,70,80,90,100},
ylabel={cpu clock time (sec)},
ymajorgrids,
title style={align=center},
title={Timing results of SD and CG\\[1ex](16 processors on 1 node)},
legend style={at={(0.03,0.97)},anchor=north west,draw=black,fill=white,legend cell align=left}
]

\addplot [
color=mycolor1,
solid,
line width=2.0pt,
mark=square,
mark options={solid}
]
table[row sep=crcr]{
10 0.45\\
20 0.03\\
30 0.06\\
40 0.21\\
50 0.56\\
60 1.48\\
70 2.21\\
80 3.91\\
90 4.87\\
100 6.6\\
110 8.48\\
120 10.8\\
130 13.91\\
140 16.98\\
150 22.21\\
160 25.41\\
170 30.64\\
180 38.56\\
};
\addlegendentry{CG with (1)};

\addplot [
color=mycolor2,
solid,
line width=2.0pt,
mark=x,
mark options={solid}
]
table[row sep=crcr]{
10 0.01\\
20 0.03\\
30 0.07\\
40 1.02\\
50 0.67\\
60 1.54\\
70 2.7\\
80 4.61\\
90 5.77\\
100 8.39\\
110 10.93\\
120 13.4\\
130 17.43\\
140 20.51\\
150 24.99\\
160 30.87\\
170 36.53\\
180 46.46\\
};
\addlegendentry{CG with (2)};

\end{axis}
\end{tikzpicture}%

    }}
}\\
\subfloat[Distributed memory]{
    \label{fig:extra_storage_costs_dist}
    \setlength\figurewidth{\linewidth} 
    \setlength\figureheight{0.80\figurewidth} 
    \hspace{-15pt}
    \resizebox{.475\textwidth}{!}{{\scalefont{1}
\definecolor{mycolor1}{rgb}{0.576470588235294,0,0}
\definecolor{mycolor2}{rgb}{1,0.647058823529412,0}

\begin{tikzpicture}

\begin{axis}[%
width=\figurewidth,
height=\figureheight,
scale only axis,
xmin=10,
xmax=180,
xtick={20,40,60,80,100,120,140,160,180},
xlabel={n},
xmajorgrids,
ymin=0,
ymax=80,
ytick={0,10,20,30,40,50,60,70,80,90,100},
ylabel={cpu clock time (sec)},
ymajorgrids,
title style={align=center},
title={Timing results of SD and CG\\[1ex](16 processors each on 24 nodes)},
legend style={at={(0.03,0.97)},anchor=north west,draw=black,fill=white,legend cell align=left}
]

\addplot [
color=mycolor1,
solid,
line width=2.0pt,
mark=square,
mark options={solid}
]
table[row sep=crcr]{
10 1.69\\
20 0.98\\
30 1.46\\
40 3.17\\
50 1.9\\
60 3.19\\
70 6.11\\
80 8.31\\
90 10.46\\
100 17.03\\
110 19.84\\
120 24.19\\
130 28.31\\
140 31.98\\
150 36.53\\
160 44.01\\
170 49.91\\
180 59.39\\
};
\addlegendentry{CG with (1)};

\addplot [
color=mycolor2,
solid,
line width=2.0pt,
mark=x,
mark options={solid}
]
table[row sep=crcr]{
10 1.13\\
20 1.49\\
30 1.79\\
40 2.02\\
50 2.23\\
60 3.88\\
70 7.56\\
80 11.26\\
90 12.29\\
100 21.14\\
110 24.33\\
120 28.22\\
130 33.04\\
140 38.6\\
150 43.25\\
160 52.41\\
170 61.33\\
180 73.58\\
};
\addlegendentry{CG with (2)};

\end{axis}
\end{tikzpicture}%

}}
}
\caption{Timing of unpreconditioned CG using~(\ref{eqn:beta_orig}) and~(\ref{eqn:beta_alt}).}
\label{fig:extra_storage_costs} 
\end{figure}

In Figure~\ref{fig:extra_storage_costs}, for the CG method without
preconditioning we see a 20-25\% cost overhead incurred due to the extra storage
and calculation for the flexible variant,~(\ref{eqn:beta_alt}), relative to the
standard variant,~(\ref{eqn:beta_orig}).  Note that in each instance of the
problem, the number of iterations of the CG method is the same regardless if
either~(\ref{eqn:beta_orig}) or~(\ref{eqn:beta_alt}) is used for $\beta_k$. 

From Figure~\ref{fig:extra_storage_costs} one observes that the distributed
memory (DM) case is not perfectly scalable, as there is about a 50\% increase
when the size of the problem is increased 24 times, proportionally to the number
of processors.  The same comparison actually holds for all other (a) and (b)
figures up to Figure~\ref{fig:LOBPCG_time}.

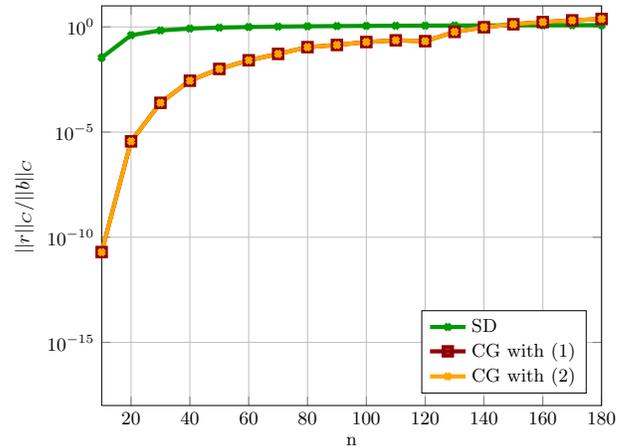
\begin{figure}[ht]
\begin{center}
    \setlength\figurewidth{\linewidth} 
    \setlength\figureheight{0.80\figurewidth} 
    \resizebox{.475\textwidth}{!}{
\definecolor{mycolor1}{rgb}{0.576470588235294,0,0}
\definecolor{mycolor2}{rgb}{1,0.647058823529412,0}

\begin{tikzpicture}

\begin{semilogyaxis}[%
width=\figurewidth,
height=\figureheight,
scale only axis,
xmin=10,
xmax=180,
xtick={20,40,60,80,100,120,140,160,180},
xlabel={n},
xmajorgrids,
ymin=1e-18,
ymax=10,
ytick={1e-15,1e-10,1e-05,1},
yminorticks=true,
ylabel={$||r||_C/||b||_C$},
ymajorgrids,
yminorgrids,
title style={align=center},
title={Final residual error of SD and CG\\[1ex](16 processors on 1 node)},
legend style={at={(0.97,0.03)},anchor=south east,draw=black,fill=white,legend cell align=left}
]
\addplot [
color=green!60!black,
solid,
line width=2.0pt,
mark=x,
mark options={solid}
]
table[row sep=crcr]{
10 0.03434329\\
20 0.4031255\\
30 0.6796543\\
40 0.8329061\\
50 0.9263677\\
60 0.9891993\\
70 1.034388\\
80 1.068504\\
90 1.095231\\
100 1.116796\\
110 1.134627\\
120 1.149678\\
130 1.162615\\
140 1.17391\\
150 1.183914\\
160 1.192885\\
170 1.201021\\
180 1.208476\\
};
\addlegendentry{SD};

\addplot [
color=mycolor1,
solid,
line width=2.0pt,
mark=square,
mark options={solid}
]
table[row sep=crcr]{
10 1.978762e-11\\
20 3.626382e-06\\
30 0.00024595\\
40 0.002766415\\
50 0.01004343\\
60 0.02588461\\
70 0.05198589\\
80 0.1091808\\
90 0.1385635\\
100 0.191013\\
110 0.2298671\\
120 0.2097425\\
130 0.5797527\\
140 0.9655748\\
150 1.338177\\
160 1.699519\\
170 2.05058\\
180 2.391975\\
};
\addlegendentry{CG with (1)};

\addplot [
color=mycolor2,
solid,
line width=2.0pt,
mark=x,
mark options={solid}
]
table[row sep=crcr]{
10 1.949569e-11\\
20 3.626382e-06\\
30 0.00024595\\
40 0.002766415\\
50 0.01004343\\
60 0.02588461\\
70 0.05198589\\
80 0.1091808\\
90 0.1385635\\
100 0.191013\\
110 0.2298671\\
120 0.2097425\\
130 0.5797527\\
140 0.9655748\\
150 1.338177\\
160 1.699519\\
170 2.05058\\
180 2.391975\\
};
\addlegendentry{CG with (2)};

\end{semilogyaxis}
\end{tikzpicture}%
    }
\end{center}
 \caption{Accuracy comparison of SD and CG using~(\ref{eqn:beta_orig})~and~(\ref{eqn:beta_alt}).}
    \label{fig:final_residual}
 \end{figure}
In Figure~\ref{fig:final_residual} we see, as expected without preconditioning,
that the relative final residuals for both CG algorithms are identical, and that
the SD algorithm of course performs much worse than either.  The number of
iterative steps is capped by 100, so most of Figure~\ref{fig:final_residual}
shows the residual after 100 iterations, except for a small straight part of the
CG line for \$n=10,20,30, where the iterations have converged. 

%

\section{Preconditioning with SMG}
\label{sec:precon_smg} 
In order to help
overcome the slow convergence of the CG method, 
as observed in Figure~\ref{fig:final_residual}, it is customary to introduce
preconditioning.  Here, we use the SMG solver as a
preconditioner, provided by  \emph{hypre}. 
The SMG solver/preconditioner uses plane-relaxation
as a smoother at each level in the V-cycle~\cite{hypreUserManual}.
The number of pre- and post-relaxation smoothing steps 
is controlled by a command line parameter in the \emph{struct}
test driver. The data in
Figures~\ref{fig:timing_1} and~\ref{fig:iterations_1} is obtained by
\begin{lstlisting}
./struct -n $n $n $n -solver 10 -v 1 1 
\end{lstlisting}
in which the  \emph{-solver 10} option refers to the SMG preconditioning, and 
the number of pre- and post-relaxation smoothing steps is specified by the  \emph{-v} 
flag---one step each of pre- and post-relaxation in this call. 

\begin{figure}[ht]
\centering
\subfloat[Shared memory]{
    \label{fig:timing_1_smp}
    \setlength\figurewidth{\linewidth} 
    \setlength\figureheight{0.80\figurewidth} 
    \resizebox{.45\textwidth}{!}{{\scalefont{1}
\definecolor{mycolor1}{rgb}{0.576470588235294,0,0}
\definecolor{mycolor2}{rgb}{1,0.647058823529412,0}

\begin{tikzpicture}

\begin{axis}[%
width=\figurewidth,
height=\figureheight,
scale only axis,
xmin=10,
xmax=180,
xtick={20,40,60,80,100,120,140,160,180},
xlabel={n},
xmajorgrids,
ymin=0,
ymax=350,
ytick={0,50,100,150,200,250,300,350},
ylabel={cpu clock time (sec)},
ymajorgrids,
title style={align=center},
title={Timing results of PCG with SMG preconditioner\\[1ex](16 processors on 1 node)},
legend style={at={(0.03,0.97)},anchor=north west,draw=black,fill=white,legend cell align=left}
]
\addplot [
color=green!60!black,
solid,
line width=2.0pt,
mark=square,
mark options={solid}
]
table[row sep=crcr]{
10 0.1\\
20 0.38\\
30 0.67\\
40 1.53\\
50 2.9\\
60 4.73\\
70 8.81\\
80 15.96\\
90 24.2\\
100 36.51\\
110 51.42\\
120 68.63\\
130 85.46\\
140 112.66\\
150 145.25\\
160 167.51\\
170 204.38\\
180 262.9\\
};
\addlegendentry{PSD-SMG};

\addplot [
color=mycolor1,
solid,
line width=2.0pt,
mark=square,
mark options={solid}
]
table[row sep=crcr]{
10 0.09\\
20 0.35\\
30 0.58\\
40 1.23\\
50 2.22\\
60 3.98\\
70 7.77\\
80 16.94\\
90 20.82\\
100 30.99\\
110 42.58\\
120 55.39\\
130 74.49\\
140 95.81\\
150 119.63\\
160 143.21\\
170 174.97\\
180 201.93\\
};
\addlegendentry{PCG-SMG with (1)};

\addplot [
color=mycolor2,
solid,
line width=2.0pt,
mark=x,
mark options={solid}
]
table[row sep=crcr]{
10 0.08\\
20 0.34\\
30 0.57\\
40 1.23\\
50 2.23\\
60 3.99\\
70 7.53\\
80 12.54\\
90 21.09\\
100 31.29\\
110 43.81\\
120 58.34\\
130 72.5\\
140 96.27\\
150 116.02\\
160 147.15\\
170 167.99\\
180 205.89\\
};
\addlegendentry{PCG-SMG with (2)};

\end{axis}
\end{tikzpicture}%
}
}
}\\
\subfloat[Distributed memory]{
    \label{fig:timing_1_dist}
    \setlength\figurewidth{\linewidth} 
    \setlength\figureheight{0.80\figurewidth} 
    \resizebox{.45\textwidth}{!}{{\scalefont{1}
\definecolor{mycolor1}{rgb}{0.576470588235294,0,0}
\definecolor{mycolor2}{rgb}{1,0.647058823529412,0}

\begin{tikzpicture}

\begin{axis}[%
width=\figurewidth,
height=\figureheight,
scale only axis,
xmin=10,
xmax=180,
xtick={20,40,60,80,100,120,140,160,180},
xlabel={n},
xmajorgrids,
ymin=0,
ymax=350,
ytick={0,50,100,150,200,250,300,350},
ylabel={cpu clock time (sec)},
ymajorgrids,
title style={align=center},
title={Timing results of PCG with SMG preconditioner\\[1ex](16 processors each on 24 nodes)},
legend style={at={(0.03,0.97)},anchor=north west,draw=black,fill=white,legend cell align=left}
]
\addplot [
color=green!60!black,
solid,
line width=2.0pt,
mark=square,
mark options={solid}
]
table[row sep=crcr]{
10 1.27\\
20 2.39\\
30 3.74\\
40 6.81\\
50 8.33\\
60 10.78\\
70 27.73\\
80 29.45\\
90 38.24\\
100 71.59\\
110 103.23\\
120 116.98\\
130 157.7\\
140 197.98\\
150 238.79\\
160 288.98\\
170 325.82\\
180 417.85\\
};
\addlegendentry{PSD-SMG};

\addplot [
color=mycolor1,
solid,
line width=2.0pt,
mark=square,
mark options={solid}
]
table[row sep=crcr]{
10 1.21\\
20 1.78\\
30 3.03\\
40 6.95\\
50 7.31\\
60 9.2\\
70 26.81\\
80 24.18\\
90 33.13\\
100 65.71\\
110 87.14\\
120 101.87\\
130 135.34\\
140 169.8\\
150 202.66\\
160 243.66\\
170 276.76\\
180 333.04\\
};
\addlegendentry{PCG-SMG with (1)};

\addplot [
color=mycolor2,
solid,
line width=2.0pt,
mark=x,
mark options={solid}
]
table[row sep=crcr]{
10 0.74\\
20 2.31\\
30 2.95\\
40 6.79\\
50 8.67\\
60 9.73\\
70 24.1\\
80 24.04\\
90 31.46\\
100 62.79\\
110 86.04\\
120 98.94\\
130 135.6\\
140 167.66\\
150 204.89\\
160 248.12\\
170 276.85\\
180 334.58\\
};
\addlegendentry{PCG-SMG with (2)};

\end{axis}
\end{tikzpicture}%
}}
}
   \caption{Cost for storage and calculation in PCG-SMG with~(\ref{eqn:beta_orig})
or~(\ref{eqn:beta_alt})}
\label{fig:timing_1}
\end{figure}
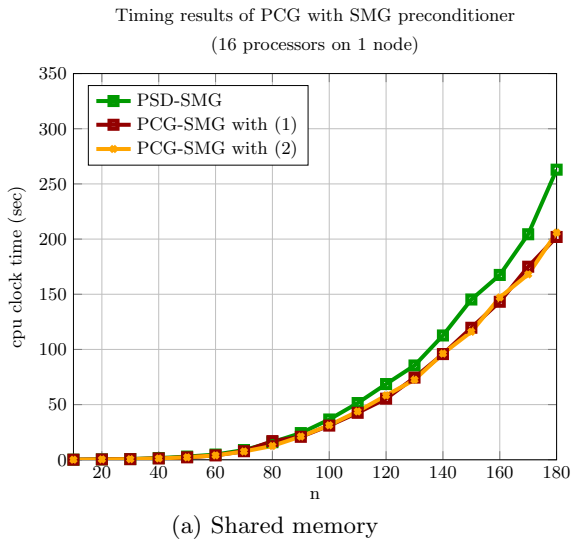
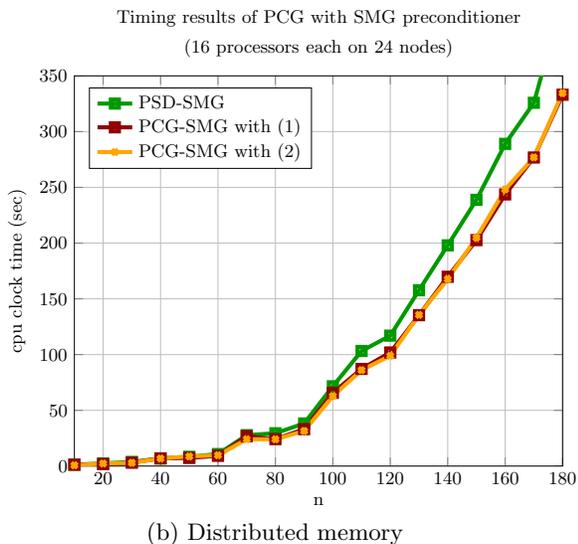
\begin{figure}[ht]
\begin{center}
    \setlength\figurewidth{\linewidth} 
    \setlength\figureheight{0.80\figurewidth} 
    \resizebox{.45\textwidth}{!}{
\definecolor{mycolor1}{rgb}{0.576470588235294,0,0}
\definecolor{mycolor2}{rgb}{1,0.647058823529412,0}

\begin{tikzpicture}

\begin{axis}[%
width=\figurewidth,
height=\figureheight,
scale only axis,
xmin=10,
xmax=180,
xtick={20,40,60,80,100,120,140,160,180},
xlabel={n},
xmajorgrids,
ymin=0,
ymax=19,
ytick={0,2,4,6,8,10,12,14,16,18},
ylabel={\# iterations},
ymajorgrids,
title style={align=center},
title={Iteration results of PCG and PSD with SMG preconditioner\\[1ex](16 processors on 1 node)},
legend style={at={(0.97,0.03)},anchor=south east,draw=black,fill=white,legend cell align=left}
]
\addplot [
color=green!60!black,
solid,
line width=2.0pt,
mark=x,
mark options={solid}
]
table[row sep=crcr]{
10 10\\
20 11\\
30 12\\
40 12\\
50 12\\
60 12\\
70 12\\
80 13\\
90 13\\
100 13\\
110 13\\
120 13\\
130 13\\
140 13\\
150 13\\
160 13\\
170 13\\
180 14\\
};
\addlegendentry{PSD-SMG};

\addplot [
color=mycolor1,
solid,
line width=2.0pt,
mark=square,
mark options={solid}
]
table[row sep=crcr]{
10 9\\
20 10\\
30 10\\
40 10\\
50 10\\
60 10\\
70 10\\
80 10\\
90 11\\
100 11\\
110 11\\
120 11\\
130 11\\
140 11\\
150 11\\
160 11\\
170 11\\
180 11\\
};
\addlegendentry{PCG-SMG with (1)};

\addplot [
color=mycolor2,
solid,
line width=2.0pt,
mark=x,
mark options={solid}
]
table[row sep=crcr]{
10 9\\
20 10\\
30 10\\
40 10\\
50 10\\
60 10\\
70 10\\
80 10\\
90 11\\
100 11\\
110 11\\
120 11\\
130 11\\
140 11\\
150 11\\
160 11\\
170 11\\
180 11\\
};
\addlegendentry{PCG-SMG with (2)};

\end{axis}
\end{tikzpicture}%
    }
\end{center}
   \caption{Iteration count for PSD-SMG and PCG-SMG with~(\ref{eqn:beta_orig})~or~(\ref{eqn:beta_alt})}
 \label{fig:iterations_1}
\end{figure}
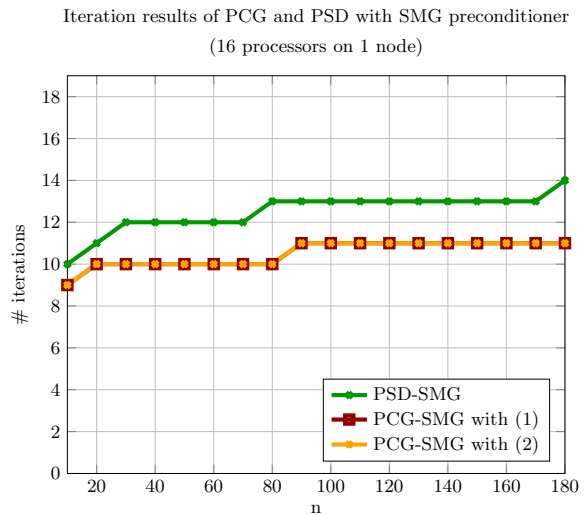

Although using~(\ref{eqn:beta_alt}) 
introduces some extra overhead, in the case of the SMG preconditioning, 
this is negligible as can be seen in
Figure~\ref{fig:timing_1}, since the SMG preconditioning, first, is
relatively expensive computationally and, second, makes the 
PCG converge much faster, see Figure~\ref{fig:iterations_1}, 
compared to the non-preconditioned case displayed in
Figure~\ref{fig:final_residual}, although the comparison is a a little confusing, since Figure~\ref{fig:final_residual} is a 
residual error plot (iterations were capped at 100) while Figure~\ref{fig:iterations_1} shows the
iteration number (convergence was obtained prior to 100 iterations).
As in Figure~\ref{fig:final_residual}, for each instance of the PCG method, the
number of iterations is equal as depicted in Figure~\ref{fig:iterations_1}.
The PSD method performs a bit worse (only 2-3 more iterations) than either variant of PCG and
not nearly as much worse as in Figure~\ref{fig:final_residual}, which clearly 
indicates the high quality of the SMG preconditioner for these problems. 

Using the same number of pre- and post-relaxation steps creates a symmetric preconditioner, 
in this test, an SPD preconditioner. 
Thus both PCG-SMG,~(\ref{eqn:beta_orig}) and~(\ref{eqn:beta_alt}),
are expected to generate identical (up to round-off errors) iterative approximations.
We indeed observe this effect in Figure~\ref{fig:iterations_1}, where the data points for 
the PCG-SMG with~(\ref{eqn:beta_orig})~and~(\ref{eqn:beta_alt})
are indistinguishable. 

If no post-relaxation is performed within the SMG preconditioner, the convergence of
the standard PCG method, i.e.,\ with~(\ref{eqn:beta_orig}), 
is significantly slowed down, almost to the level where
preconditioning is useless.  The command 
used for this comparison is:
\begin{lstlisting}
./struct -n 80 80 80 -solver 10 -v 1 0
\end{lstlisting}
in which the unbalanced relaxation is specified by the \emph{-v} flag
and the grid is 1280x80x80. 
In Figure~\ref{fig:iters_80}, we compare convergence 
of PCG-SMG using~(\ref{eqn:beta_orig})~and~(\ref{eqn:beta_alt}) and
PSD-SMG, where the 
SMG preconditioner has no post-relaxation.  
In PCG-SMG using~(\ref{eqn:beta_alt}) without post-relaxation of the SMG preconditioner, we
still achieve a very good convergence rate, as well as for PSD-SMG.  

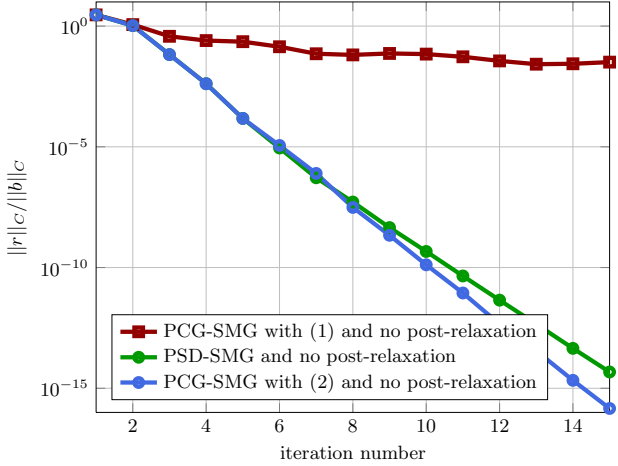
\begin{figure}[ht]
\begin{center}
    \setlength\figurewidth{\linewidth} 
    \setlength\figureheight{0.80\figurewidth} 
    \hspace{-5ex}
    \resizebox{.475\textwidth}{!}{
\definecolor{mycolor1}{rgb}{0.576470588235294,0,0}
\definecolor{mycolor2}{rgb}{0.254901960784314,0.411764705882353,0.882352941176471}

\begin{tikzpicture}

\begin{semilogyaxis}[%
width=\figurewidth,
height=\figureheight,
scale only axis,
xmin=1,
xmax=15,
xtick={2,4,6,8,10,12,14},
xlabel={iteration number},
xmajorgrids,
ymin=1e-16,
ymax=10,
ytick={1e-15,1e-10,1e-05,1},
yminorticks=true,
ylabel={$||r||_C/||b||_C$},
ymajorgrids,
yminorgrids,
title style={align=center},
title={Iteration error of PCG and PSD with SMG preconditioner\\[1ex](16 processors on 1 node)},
legend style={at={(0.03,0.03)},anchor=south west,draw=black,fill=white,legend cell align=left}
]
\addplot [
color=mycolor1,
solid,
line width=2.0pt,
mark=square,
mark options={solid}
]
table[row sep=crcr]{
1 2.913096\\
2 1.151389\\
3 0.3773084\\
4 0.2513774\\
5 0.2261113\\
6 0.1393729\\
7 0.07144089\\
8 0.06408111\\
9 0.07352382\\
10 0.06911076\\
11 0.05345101\\
12 0.03612713\\
13 0.02640552\\
14 0.02748725\\
15 0.03218463\\
16 0.03496089\\
};
\addlegendentry{PCG-SMG with (1) and no post-relaxation};

\addplot [
color=green!60!black,
solid,
line width=2.0pt,
mark=o,
mark options={solid}
]
table[row sep=crcr]{
1 2.913096\\
2 1.043479\\
3 0.06613261\\
4 0.004120495\\
5 0.0001500428\\
6 8.938381e-06\\
7 5.220239e-07\\
8 5.163742e-08\\
9 4.525627e-09\\
10 4.633801e-10\\
11 4.484943e-11\\
12 4.461817e-12\\
13 4.559614e-13\\
14 4.495435e-14\\
15 4.728019e-15\\
16 4.663512e-16\\
};
\addlegendentry{PSD-SMG and no post-relaxation};

\addplot [
color=mycolor2,
solid,
line width=2.0pt,
mark=o,
mark options={solid}
]
table[row sep=crcr]{
1 2.913096\\
2 1.043637\\
3 0.06587817\\
4 0.00409623\\
5 0.0001497729\\
6 1.13633e-05\\
7 7.901603e-07\\
8 3.088125e-08\\
9 2.162154e-09\\
10 1.297309e-10\\
11 8.829039e-12\\
12 4.558494e-13\\
13 3.5472e-14\\
14 2.132461e-15\\
15 1.458223e-16\\
};
\addlegendentry{PCG-SMG with (2) and no post-relaxation};

\end{semilogyaxis}
\end{tikzpicture}%
    }
\end{center}
\caption{Iteration Error for PCG-SMG and PSD-SMG}
\label{fig:iters_80}
\end{figure}  

The convergence behavior in Figure~\ref{fig:iters_80} is similar to that 
observed in~\cite{knyazev1267}.
There, a variable SPD preconditioner makes 
the standard PCG, i.e.,\ using (\ref{eqn:beta_orig}), almost stall, while 
the convergence rates of PSD and flexible PCG, 
i.e.,\ with ~(\ref{eqn:beta_alt}), 
are good and close to each other.

However, the SMG preconditioner is fixed and linear, according to its 
description in~\cite{schaffer305237} and our numerical verification, 
so the theoretical explanation of such a convergence behavior in~\cite{knyazev1267}
is not directly applicable here. Moreover, turning off the post-relaxation smoothing in 
a multigrid preconditioner makes it nonsymmetric---the case not 
theoretically covered in~\cite{knyazev1267}, 
where the assumption is always made that the preconditioner is SPD. 

\begin{figure}[ht]
\centering
\subfloat[Shared memory]{
    \label{fig:timing_smp}
    \setlength\figurewidth{\linewidth} 
    \setlength\figureheight{0.80\figurewidth} 
    \hspace{-15pt}
    \resizebox{.475\textwidth}{!}{{\scalefont{1}
\definecolor{mycolor1}{rgb}{0.576470588235294,0,0}
\definecolor{mycolor2}{rgb}{1,0.647058823529412,0}
\definecolor{mycolor3}{rgb}{0.254901960784314,0.411764705882353,0.882352941176471}

\begin{tikzpicture}

\begin{axis}[%
width=\figurewidth,
height=\figureheight,
scale only axis,
xmin=10,
xmax=180,
xtick={20,40,60,80,100,120,140,160,180},
xlabel={n},
xmajorgrids,
ymin=0,
ymax=350,
ytick={0,50,100,150,200,250,300,350},
ylabel={cpu clock time (sec)},
ymajorgrids,
title style={align=center},
title={Timing results of PCG with SMG preconditioner\\[1ex](16 processors on 1 node)},
legend style={at={(0.03,0.97)},anchor=north west,draw=black,fill=white,legend cell align=left}
]
\addplot [
color=mycolor1,
solid,
line width=2.0pt,
mark=square,
mark options={solid}
]
table[row sep=crcr]{
10 0.09\\
20 0.35\\
30 0.58\\
40 1.23\\
50 2.22\\
60 3.98\\
70 7.77\\
80 16.94\\
90 20.82\\
100 30.99\\
110 42.58\\
120 55.39\\
130 74.49\\
140 95.81\\
150 119.63\\
160 143.21\\
170 174.97\\
180 201.93\\
};
\addlegendentry{PCG-SMG with (1) and balanced relaxation};

\addplot [
color=mycolor2,
solid,
line width=2.0pt,
mark=x,
mark options={solid}
]
table[row sep=crcr]{
10 0.08\\
20 0.34\\
30 0.57\\
40 1.23\\
50 2.23\\
60 3.99\\
70 7.53\\
80 12.54\\
90 21.09\\
100 31.29\\
110 43.81\\
120 58.34\\
130 72.5\\
140 96.27\\
150 116.02\\
160 147.15\\
170 167.99\\
180 205.89\\
};
\addlegendentry{PCG-SMG with (2) and balanced relaxation};

\addplot [
color=mycolor3,
solid,
line width=2.0pt,
mark=o,
mark options={solid}
]
table[row sep=crcr]{
10 0.04\\
20 0.17\\
30 0.3\\
40 0.64\\
50 1.3\\
60 2.41\\
70 4.53\\
80 7.51\\
90 11.28\\
100 17.58\\
110 22.88\\
120 35.44\\
130 40.75\\
140 54.04\\
150 66.91\\
160 82.37\\
170 96.19\\
180 117.89\\
};
\addlegendentry{PCG-SMG with (2) and no post-relaxation};

\end{axis}
\end{tikzpicture}%
}}
}\\
\vspace{4pt}
\subfloat[Distributed memory]{
    \label{fig:timing_dist}
    \setlength\figurewidth{\linewidth} 
    \setlength\figureheight{0.80\figurewidth} 
    \hspace{-15pt}
    \resizebox{.475\textwidth}{!}{{\scalefont{1}
\definecolor{mycolor1}{rgb}{0.576470588235294,0,0}
\definecolor{mycolor2}{rgb}{1,0.647058823529412,0}
\definecolor{mycolor3}{rgb}{0.254901960784314,0.411764705882353,0.882352941176471}

\begin{tikzpicture}

\begin{axis}[%
width=\figurewidth,
height=\figureheight,
scale only axis,
xmin=10,
xmax=180,
xtick={20,40,60,80,100,120,140,160,180},
xlabel={n},
xmajorgrids,
ymin=0,
ymax=350,
ytick={0,50,100,150,200,250,300,350},
ylabel={cpu clock time (sec)},
ymajorgrids,
title style={align=center},
title={Timing results of PCG with SMG preconditioner\\[1ex](16 processors each on 24 nodes)},
legend style={at={(0.03,0.97)},anchor=north west,draw=black,fill=white,legend cell align=left}
]
\addplot [
color=mycolor1,
solid,
line width=2.0pt,
mark=square,
mark options={solid}
]
table[row sep=crcr]{
10 1.21\\
20 1.78\\
30 3.03\\
40 6.95\\
50 7.31\\
60 9.2\\
70 26.81\\
80 24.18\\
90 33.13\\
100 65.71\\
110 87.14\\
120 101.87\\
130 135.34\\
140 169.8\\
150 202.66\\
160 243.66\\
170 276.76\\
180 333.04\\
};
\addlegendentry{PCG-SMG with (1) and balanced relaxation};

\addplot [
color=mycolor2,
solid,
line width=2.0pt,
mark=x,
mark options={solid}
]
table[row sep=crcr]{
10 0.74\\
20 2.31\\
30 2.95\\
40 6.79\\
50 8.67\\
60 9.73\\
70 24.1\\
80 24.04\\
90 31.46\\
100 62.79\\
110 86.04\\
120 98.94\\
130 135.6\\
140 167.66\\
150 204.89\\
160 248.12\\
170 276.85\\
180 334.58\\
};
\addlegendentry{PCG-SMG with (2) and balanced relaxation};

\addplot [
color=mycolor3,
solid,
line width=2.0pt,
mark=o,
mark options={solid}
]
table[row sep=crcr]{
10 1.03\\
20 1.61\\
30 1.86\\
40 4.11\\
50 4.83\\
60 5.54\\
70 15.94\\
80 14.92\\
90 17.97\\
100 34.92\\
110 45.79\\
120 54.33\\
130 72.13\\
140 89.92\\
150 112.42\\
160 137.45\\
170 155.71\\
180 185.83\\
};
\addlegendentry{PCG-SMG with (2) and no post-relaxation};

\end{axis}
\end{tikzpicture}%
}}
}
\caption{Cost comparison of relaxation in PCG-SMG using~(\ref{eqn:beta_orig}) 
  or~(\ref{eqn:beta_alt})}
\label{fig:timing}
\end{figure}
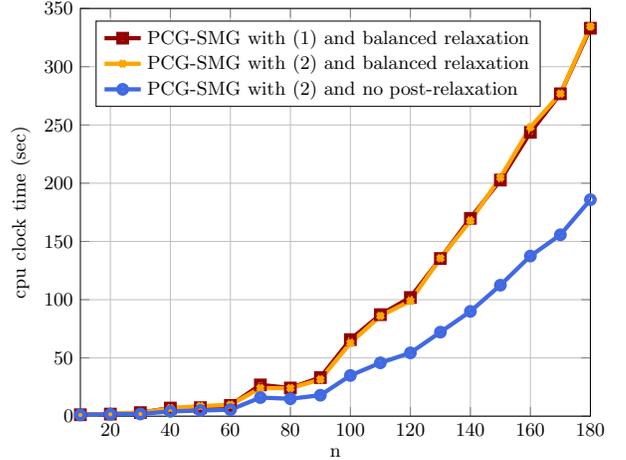

\begin{figure}[ht] 
\begin{center}
    \setlength\figurewidth{\linewidth} 
    \setlength\figureheight{0.80\figurewidth} 
    \resizebox{.475\textwidth}{!}{
\definecolor{mycolor1}{rgb}{0.254901960784314,0.411764705882353,0.882352941176471}
\definecolor{mycolor2}{rgb}{0.576470588235294,0,0}
\definecolor{mycolor3}{rgb}{1,0.647058823529412,0}

\begin{tikzpicture}

\begin{axis}[%
width=\figurewidth,
height=\figureheight,
scale only axis,
xmin=10,
xmax=180,
xtick={20,40,60,80,100,120,140,160,180},
xlabel={n},
xmajorgrids,
ymin=0,
ymax=19,
ytick={0,2,4,6,8,10,12,14,16,18},
ylabel={\# iterations},
ymajorgrids,
title style={align=center},
title={Iteration results of PCG and PSD with SMG preconditioner\\[1ex](16 processors on 1 node)},
legend style={at={(0.97,0.03)},anchor=south east,draw=black,fill=white,legend cell align=left}
]
\addplot [
color=green!60!black,
solid,
line width=2.0pt,
mark=o,
mark options={solid}
]
table[row sep=crcr]{
10 13\\
20 14\\
30 15\\
40 15\\
50 16\\
60 16\\
70 16\\
80 16\\
90 16\\
100 17\\
110 17\\
120 17\\
130 17\\
140 17\\
150 18\\
160 18\\
170 18\\
180 18\\
};
\addlegendentry{PSD-SMG with no post-relaxation};

\addplot [
color=mycolor1,
solid,
line width=2.0pt,
mark=o,
mark options={solid}
]
table[row sep=crcr]{
10 12\\
20 13\\
30 14\\
40 14\\
50 15\\
60 15\\
70 15\\
80 15\\
90 15\\
100 15\\
110 15\\
120 15\\
130 16\\
140 16\\
150 16\\
160 16\\
170 16\\
180 16\\
};
\addlegendentry{PCG-SMG with (2) and no post-relaxation};

\addplot [
color=mycolor2,
solid,
line width=2.0pt,
mark=square,
mark options={solid}
]
table[row sep=crcr]{
10 9\\
20 10\\
30 10\\
40 10\\
50 10\\
60 10\\
70 10\\
80 10\\
90 11\\
100 11\\
110 11\\
120 11\\
130 11\\
140 11\\
150 11\\
160 11\\
170 11\\
180 11\\
};
\addlegendentry{PCG-SMG with (1) and balanced relaxation};

\addplot [
color=mycolor3,
solid,
line width=2.0pt,
mark=x,
mark options={solid}
]
table[row sep=crcr]{
10 9\\
20 10\\
30 10\\
40 10\\
50 10\\
60 10\\
70 10\\
80 10\\
90 11\\
100 11\\
110 11\\
120 11\\
130 11\\
140 11\\
150 11\\
160 11\\
170 11\\
180 11\\
};
\addlegendentry{PCG-SMG with (2) and balanced relaxation};

\end{axis}
\end{tikzpicture}%
    }
    \vspace{2pt}
\end{center}
\caption{Iteration count for PSD-SMG and PCG-SMG using~(\ref{eqn:beta_orig}) 
  or~(\ref{eqn:beta_alt})}
\label{fig:iterations}
\end{figure}
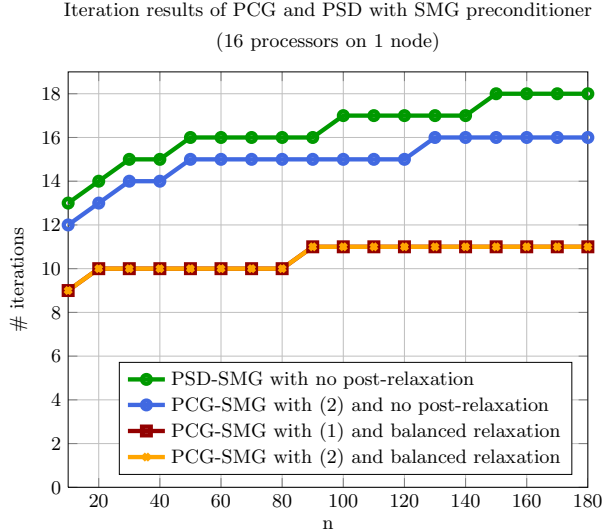 

We add the data for the absence of post-relaxation (option \emph{-v 1 0}) 
in the SMG preconditioner to 
Figures~\ref{fig:timing_1} and~\ref{fig:iterations_1} to obtain 
Figures~\ref{fig:timing} and~\ref{fig:iterations}. The number of iterations of the 
solver does increase a bit for both~(\ref{eqn:beta_alt}) and for the PSD method with
no post-relaxation in the SMG preconditioner, 
but not enough to outweigh the cost savings of not using 
the post-relaxation in the SMG preconditioner. 
The overall improvement is 43\% for all problems tested.


\section{Preconditioning of LOBPCG\\ with SMG}
\label{sec:precon_lobpcg_smg}
The LOBPCG method~\cite{knyazev:517} computes the $m$ smallest eigenvalues of a
linear operator and is implemented within \emph{hypre} through BLOPEX; see~\cite{knyazev2224}.
We conclude our numerical experiments with a comparison of the use of balanced and 
unbalanced relaxation in the SMG preconditioner for the LOBPCG method with $m=1.$

In these tests, the matrix remains the same as in the previous section, i.e.,
corresponds to the three-dimensional negative Laplacian in a brick with
homogeneous Dirichlet boundary conditions approximated by the standard finite
difference scheme using the 7-point stencil with the grid size of one in all
three directions. But in this section we compute the smallest eigenvalue and the
corresponding eigenvector, rather than solve a linear system.  The initial
approximation to the eigenvector is (pseudo)random. 

We generate the data for Figures~\ref{fig:LOBPCG_time}~and~\ref{fig:LOBPCG_itrs} 
with the commands:
\begin{lstlisting}
./struct -n $n $n $n -solver 10 -lobpcg -v 1 0
./struct -n $n $n $n -solver 10 -lobpcg -v 1 1
\end{lstlisting}
where \$n runs from 10 to 120.  We also use the \emph{-P} option in $struct$ to
make the brick more even-sided.  In the shared memory experiment we use
\emph{-P}
$4$ $2$ $2$ which will create a 4$n$-by-2$n$-by-2$n$ brick and in the distrubuted
case we use \emph{-P} $8$ $8$ $6$ which will create an 8$n$-by-8$n$-by-6$n$ brick.

For this experiment, as seen in Figure~\ref{fig:LOBPCG_itrs}, the number of
LOBPCG iterations for the non-balanced SMG preconditioner is roughly 50\% more
than that for the balanced SMG preconditioner.  However, the cost savings, which
is about 30-50\%, of not using the post-relaxation in the SMG preconditioner
outweigh the higher number of iterations, as shown in
Figure~\ref{fig:LOBPCG_time}, and thus justify turning off the post-relaxation
in this case.  From Figure~\ref{fig:LOBPCG_time} one observes that the
distributed memory case is not perfectly scalable, as there is about a 50\%
increase when the size of the problem is increased 24 times, proportionally to
the number of processors.  Finally, we note that the existing LOBPCG convergence
theory in~\cite{knyazev:517,MR2530268} requires an SPD preconditioner $T$, and
does not explain convergence if $T$ is nonsymmetric.  

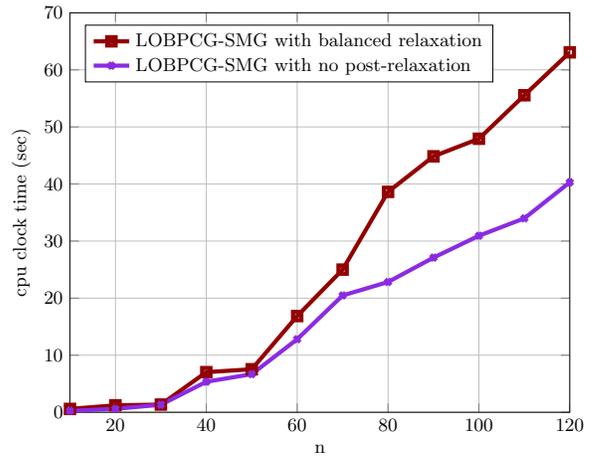
\begin{figure}[ht]
\centering
\subfloat[Shared memory]{
    \label{fig:LOBPCG_time_smp}
    \setlength\figurewidth{\linewidth} 
    \setlength\figureheight{0.80\figurewidth} 
    \resizebox{.45\textwidth}{!}{{\scalefont{1}
\definecolor{mycolor1}{rgb}{0.576470588235294,0,0}%
\definecolor{mycolor2}{rgb}{0.541176470588235,0.168627450980392,0.886274509803922}%

\begin{tikzpicture}

\begin{axis}[%
width=\figurewidth,
height=\figureheight,
scale only axis,
xmin=10,
xmax=120,
xtick={ 20,  40,  60,  80, 100, 120},
xlabel={n},
xmajorgrids,
ymin=0,
ymax=70,
ytick={ 0, 10, 20, 30, 40, 50, 60, 70},
ylabel={cpu clock time (sec)},
ymajorgrids,
title style={align=center},
title={Timing of LOBPCG-SMG with and without balanced relaxation\\[1ex](16 Processors on 1 node)},
legend style={at={(0.03,0.97)},anchor=north west,draw=black,fill=white,legend cell align=left}
]
\addplot [
color=mycolor1,
solid,
line width=2.0pt,
mark=square,
mark options={solid}
]
table[row sep=crcr]{
10 0.2\\
20 0.39\\
30 0.59\\
40 1.25\\
50 1.61\\
60 3.68\\
70 5.29\\
80 8.66\\
90 13\\
100 21.28\\
110 26.61\\
120 34.96\\
};
\addlegendentry{LOBPCG-SMG with balanced relaxation};

\addplot [
color=mycolor2,
solid,
line width=2.0pt,
mark=x,
mark options={solid}
]
table[row sep=crcr]{
10 0.08\\
20 0.17\\
30 0.29\\
40 0.52\\
50 0.95\\
60 1.7\\
70 3.1\\
80 6.1\\
90 7.67\\
100 10.92\\
110 16.33\\
120 17.44\\
};
\addlegendentry{LOBPCG-SMG with no post-relaxation};

\end{axis}
\end{tikzpicture}%
}}
}\\
\subfloat[Distributed memory]{
    \label{fig:LOBPCG_time_dist}
    \setlength\figurewidth{\linewidth} 
    \setlength\figureheight{0.80\figurewidth} 
    \resizebox{.45\textwidth}{!}{{\scalefont{1}
\definecolor{mycolor1}{rgb}{0.576470588235294,0,0}%
\definecolor{mycolor2}{rgb}{0.541176470588235,0.168627450980392,0.886274509803922}%

\begin{tikzpicture}

\begin{axis}[%
width=\figurewidth,
height=\figureheight,
scale only axis,
xmin=10,
xmax=120,
xtick={ 20,  40,  60,  80, 100, 120},
xlabel={n},
xmajorgrids,
ymin=0,
ymax=70,
ytick={ 0, 10, 20, 30, 40, 50, 60, 70},
ylabel={cpu clock time (sec)},
ymajorgrids,
title style={align=center},
title={Timing of LOBPCG-SMG with and without balanced relaxation\\[1ex](16 processors each on 24 nodes)},
legend style={at={(0.03,0.97)},anchor=north west,draw=black,fill=white,legend cell align=left}
]
\addplot [
color=mycolor1,
solid,
line width=2.0pt,
mark=square,
mark options={solid}
]
table[row sep=crcr]{
10 0.59\\
20 1.22\\
30 1.35\\
40 7.03\\
50 7.54\\
60 16.86\\
70 25\\
80 38.62\\
90 44.85\\
100 47.93\\
110 55.55\\
120 63.07\\
};
\addlegendentry{LOBPCG-SMG with balanced relaxation};

\addplot [
color=mycolor2,
solid,
line width=2.0pt,
mark=x,
mark options={solid}
]
table[row sep=crcr]{
10 0.24\\
20 0.6\\
30 1.32\\
40 5.35\\
50 6.68\\
60 12.8\\
70 20.47\\
80 22.83\\
90 27.1\\
100 30.93\\
110 33.98\\
120 40.3\\
};
\addlegendentry{LOBPCG-SMG with no post-relaxation};

\end{axis}
\end{tikzpicture}%
}}
}
\caption{Cost for higher relaxation level for LOBPCG-SMG}
\label{fig:LOBPCG_time}
\end{figure}

\begin{figure}[ht]
\centering
\subfloat[Shared memory]{
    \label{fig:LOBPCG_itrs_smp}
    \setlength\figurewidth{\linewidth} 
    \setlength\figureheight{0.80\figurewidth} 
    \hspace{-15pt}
    \resizebox{.475\textwidth}{!}{{\scalefont{1}
\definecolor{mycolor1}{rgb}{0.576470588235294,0,0}%
\definecolor{mycolor2}{rgb}{0.541176470588235,0.168627450980392,0.886274509803922}%

\begin{tikzpicture}

\begin{axis}[%
width=\figurewidth,
height=\figureheight,
scale only axis,
xmin=10,
xmax=120,
xtick={ 20,  40,  60,  80, 100, 120},
xlabel={n},
xmajorgrids,
ymin=0,
ymax=15,
ytick={ 0,  1,  2,  3,  4,  5,  6,  7,  8,  9, 10, 11, 12, 13, 14, 15},
ylabel={\# iterations},
ymajorgrids,
title style={align=center},
title={Iteration count of LOBPCG-SMG with and without balanced relaxation\\[1ex](16 Processors on 1 node)},
legend style={at={(0.03,0.97)},anchor=north west,draw=black,fill=white,legend cell align=left}
]
\addplot [
color=mycolor1,
solid,
line width=2.0pt,
mark=square,
mark options={solid}
]
table[row sep=crcr]{
10 9\\
20 10\\
30 9\\
40 10\\
50 7\\
60 9\\
70 7\\
80 7\\
90 7\\
100 8\\
110 7\\
120 7\\
};
\addlegendentry{LOBPCG-SMG with balanced relaxation};

\addplot [
color=mycolor2,
solid,
line width=2.0pt,
mark=x,
mark options={solid}
]
table[row sep=crcr]{
10 11\\
20 12\\
30 12\\
40 9\\
50 10\\
60 9\\
70 9\\
80 11\\
90 9\\
100 9\\
110 10\\
120 8\\
};
\addlegendentry{LOBPCG-SMG with no post-relaxation};

\end{axis}
\end{tikzpicture}%
}}
}\\
\vspace{-2pt}
\subfloat[Distributed memory]{
    \label{fig:LOBPCG_itrs_dist}
    \setlength\figurewidth{\linewidth} 
    \setlength\figureheight{0.80\figurewidth} 
    \hspace{-15pt}
    \resizebox{.475\textwidth}{!}{{\scalefont{1}
\definecolor{mycolor1}{rgb}{0.576470588235294,0,0}%
\definecolor{mycolor2}{rgb}{0.541176470588235,0.168627450980392,0.886274509803922}%

\begin{tikzpicture}

\begin{axis}[%
width=\figurewidth,
height=\figureheight,
scale only axis,
xmin=10,
xmax=120,
xtick={ 20,  40,  60,  80, 100, 120},
xlabel={n},
xmajorgrids,
ymin=0,
ymax=10,
ytick={ 0,  1,  2,  3,  4,  5,  6,  7,  8,  9, 10},
ylabel={\# iterations},
ymajorgrids,
title style={align=center},
title={Iteration count of LOBPCG-SMG with and without balanced relaxation\\[1ex](16 processors each on 24 nodes)},
legend style={draw=black,fill=white,legend cell align=left}
]
\addplot [
color=mycolor1,
solid,
line width=2.0pt,
mark=square,
mark options={solid}
]
table[row sep=crcr]{
10 7\\
20 7\\
30 6\\
40 6\\
50 5\\
60 5\\
70 5\\
80 6\\
90 5\\
100 5\\
110 5\\
120 5\\
};
\addlegendentry{LOBPCG-SMG with balanced relaxation};

\addplot [
color=mycolor2,
solid,
line width=2.0pt,
mark=x,
mark options={solid}
]
table[row sep=crcr]{
10 8\\
20 9\\
30 8\\
40 7\\
50 7\\
60 7\\
70 7\\
80 7\\
90 7\\
100 7\\
110 7\\
120 7\\
};
\addlegendentry{LOBPCG-SMG with no post-relaxation};

\end{axis}
\end{tikzpicture}%
}}
}
\caption{Iteration comparison of relaxation levels for LOBPCG-SMG }
\label{fig:LOBPCG_itrs}
\end{figure}
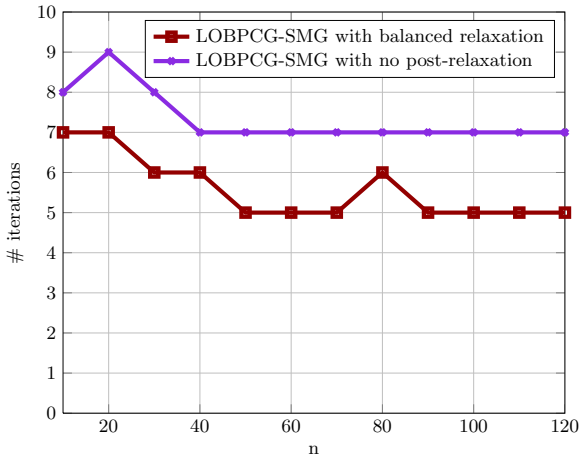

\section{Theoretical justification of\\ the observed numerical\\ behavior}
\label{sec:theory}
For linear systems with SPD coefficient matrices, the use of nonsymmetric 
preconditioning has been justified, e.g., in
\cite[Section 12.3]{MR1276069},
\cite{MR1934875,MR1797890}, and 
\cite[Section 10.2]{MR2427040}. 
The arguments used there are applicable for nonsymmetric and variable preconditioning. 

Specifically, it is shown in \cite[Section 10.2]{MR2427040} that the 
flexible  PCG, i.e.,\ using~(\ref{eqn:beta_alt}), is locally optimal, i.e.,\ 
on every step it converges not slower than PSD. The convergence rate bound for the 
PSD with nonsymmetric preconditioning established  in \cite[Section 10.2]{MR2427040} is
\begin{equation} 
\|r_{k+1}\|_{A^{-1}}\leq \delta \|r_{k}\|_{A^{-1}},
\label{eqn:Vb}
\end{equation}
under the assumption 
\begin{equation} 
\|I-AT\|_{A^{-1}}\leq \delta<1,
\label{eqn:Va}
\end{equation} 
where $\|\cdot\|_{A^{-1}}$ denotes the operator norm induced 
by the corresponding vector norm $\|x\|_{A^{-1}}=\sqrt{x\rq{}A^{-1}x}$.

The key identity for PSD, that easily leads to bound \eqref{eqn:Vb}, is presented in the following theorem.
\begin{theorem}\label{thm:Vi}
The identity holds,
\begin{equation} 
\|r_{k+1}\|_{A^{-1}} / \|r_{k}\|_{A^{-1}}=
\sin\left(\angle_{A^{-1}}\{r_{k}, AT r_{k}\}\right),
\label{eqn:Vi}
\end{equation} 
where the right-hand side is defined via
\[
\cos\left(\angle_{A^{-1}}\{r_{k}, AT r_{k}\}\right)=
\frac{\left|(r_{k})' T r_{k}\right|}{ \|r_{k}\|_{A^{-1}}\|ATr_{k}\|_{A^{-1}}}.
\]
\end{theorem}
\begin{proof}
Identity  \eqref{eqn:Vi} is actually proved, although not explicitly formulated, 
in the proof of \cite[Theorem 10.2]{MR2427040}. Alternatively, 
identity  \eqref{eqn:Vi} is equivalent to 
\begin{equation} 
\|e_{k+1}\|_{A} / \|e_{k}\|_{A}=
\sin\left(\angle_{A}\{e_{k}, TA e_{k}\}\right),
\label{eqn:KLi}
\end{equation}
where $A e_{k}=r_{k}$, which is the statement of \cite[Lemma 4.1]{knyazev1267}.
We note that  \cite{knyazev1267} generally assumes that the preconditioner $T$ is SPD, 
but this assumption is not actually used in the proof of  \cite[Lemma 4.1]{knyazev1267}.
\end{proof}

Assumption \eqref{eqn:Va} is very simple, but has one significant drawback---it does not 
allow arbitrary scaling of the preconditioner $T$, while the PCG and PSD methods are 
invariant with respect to scaling of $T.$ The way around it is to scale the  preconditioner $T$
\emph{before} assumption \eqref{eqn:Va} is verified.  We now illustrate such a scaling 
under an additional assumption that $T$ is SPD, following \cite{knyazev1267}. 
We start with a theorem, connecting assumption \eqref{eqn:Va} with its equivalent, 
and probably more  traditional form.
\begin{theorem}\label{thm:Vii}
Let the  preconditioner $T$ be SPD. Then assumption  \eqref{eqn:Va} is equivalent to
\begin{equation} 
\|I-TA\|_{T^{-1}}\leq \delta<1.
\label{eqn:KLa}
\end{equation} 
\end{theorem}
\begin{proof}
Since $T$ is SPD, on the one hand, the matrix product $AT$ is also SPD, but with respect to the $A^{-1}$ scalar product. 
This implies that assumption  \eqref{eqn:Va} is equivalent to the statement that 
$\Lambda(AT)\in[1-\delta,1+\delta]$ with $\delta<1,$
where $\Lambda(\cdot)$ denotes the matrix spectrum.
On the other hand,  the matrix product $TA$ is  SPD as well, with respect to the $T^{-1}$ scalar product. 
Thus,  assumption  \eqref{eqn:KLa} is equivalent to the statement that 
$\Lambda(TA)\in[1-\delta,1+\delta]$. 
This means the equivalence of assumptions  \eqref{eqn:Va} and  \eqref{eqn:KLa}, since 
 $\Lambda(AT)=\Lambda(TA)$.
\end{proof}

Let us now, without loss of generality, as in \cite[p. 96]{MR1942725} and 
\cite[pp. 1268--1269]{knyazev1267}, 
always scale the SPD preconditioner $T$ in such a way that 
\[
\max\{\Lambda(TA)\}+\min\{\Lambda(TA)\}=2.
\]
Then we have $\delta=(\kappa(TA)-1)/(\kappa(TA)+1)$  and, vice versa, 
$\kappa(TA)=(1+\delta)/(1-\delta)$, where $\kappa(\cdot)$ denotes the matrix spectral condition number. 
The convergence rate bound \eqref{eqn:Vb} for the PSD with nonsymmetric preconditioning in this case 
turns into the standard PSD  convergence rate bound for the case of SPD preconditioner $T$; 
see. e.g.,\  \cite[Bound (1.3)]{knyazev1267}. 
Moreover,  \cite[Theorem 5.1]{knyazev1267} shows that this convergence rate bound is sharp 
for PSD, and cannot be improved for flexible  PCG, i.e.,\ using~(\ref{eqn:beta_alt}),
if the SPD preconditioner $T$ changes on every iteration. The latter result  
naturally extends to the case of nonsymmetric preconditioning of \cite[Section 10.2]{MR2427040}. 

Compared to linear systems, eigenvalue problems are 
significantly more complex. Sharp convergence rate bounds 
for symmetric eigenvalue problems have been obtained in the last decade, 
and only for the simplest preconditioned method; 
see \cite{MR1942725,MR2530268} and references therein. 
A possibility of using nonsymmetric preconditioning 
for symmetric eigenvalue problems has not been considered before, 
to our knowledge.  However, our check of arguments of 
\cite{MR1942725} and preceding works, where a PSD convergence rate bound is proved  
assuming  \eqref{eqn:Va} and SPD preconditioning, reveals 
that the latter assumption, SPD, is actually never significantly used, 
and can be dropped without affecting the bound. 

The arguments above lead us to a surprising determination 
that whether or not the preconditioner is SPD is of no importance for PSD
convergence, given the same quality of preconditioning, measured by
\eqref{eqn:Va} after preconditioner prescaling.  If the preconditioner is fixed
SPD then the standard PCG is the method of choice.  The cases, where the
preconditioner is variable or nonsymmetric, are similar to each other---the
standard non-flexible  PCG, i.e.,\ using~(\ref{eqn:beta_orig}), stalls, while
the flexible PCG converges, due to its local optimality, but may not be much
faster compared to PSD.  This explains the numerical results using nonsymmetric
preconditioning reported  in this work, as related to results of
\cite{knyazev1267} for variable SPD preconditioning. 

\section{Conclusion} 
Although the flexible PCG linear solver does require a bit more computational
effort and storage as compared to the standard PCG, within the scope of
preconditioning the extra effort can be worthwhile, if the preconditioner is not
fixed SPD.  Moreover, our numerical tests showed that the behavior is similar in
both shared and distributed memory.  Thus our conclusion is that the effect of
acceleration is independent of the speed of the node interconnection.  The use
of geometric multigrid without post-relaxation is demonstrated to be
surprisingly efficient as a preconditioner for locally optimal iterative
methods, such as the flexible PCG for linear systems and LOBPCG for eigenvalue
problems.    

\section{Acknowledgment}
The authors would like to thank Rob Falgout, Van Henson, Panayot Vassilevski,
and other members of the \emph{hypre} team for their attention to our numerical
results reported here and the Center for Computational Mathematics University of
Colorado Denver for the use of the cluster.  This work is partially supported by
NSF awards CNS 0958354 and DMS 1115734.

\bibliographystyle{abbrv}
\bibliography{biblio}

%


\end{document}